\documentclass[english]{article}
\usepackage[T1]{fontenc}
\usepackage[latin9]{inputenc}
\usepackage{geometry}
\usepackage{amsmath}
\usepackage{amsthm}
\usepackage{amssymb}

\makeatletter
\numberwithin{equation}{section}
\numberwithin{figure}{section}
\theoremstyle{plain}
\newtheorem{thm}{\protect\theoremname}
\theoremstyle{definition}
\newtheorem{defn}[thm]{\protect\definitionname}
\theoremstyle{plain}
\newtheorem{lem}[thm]{\protect\lemmaname}

\usepackage{hyperref}
\hypersetup{
    colorlinks,
    allcolors=red
}
\usepackage[lined,boxed,ruled,norelsize,algo2e]{algorithm2e}
\@addtoreset{section}{part}

\makeatother

\usepackage{babel}
\providecommand{\definitionname}{Definition}
\providecommand{\lemmaname}{Lemma}
\providecommand{\theoremname}{Theorem}

\begin{document}
\global\long\def\defeq{\stackrel{\mathrm{{\scriptscriptstyle def}}}{=}}%
\global\long\def\norm#1{\left\Vert #1\right\Vert }%
\global\long\def\R{\mathbb{R}}%
 
\global\long\def\Rn{\mathbb{R}^{n}}%
\global\long\def\tr{\mathrm{Tr}}%
\global\long\def\diag{\mathrm{diag}}%
\global\long\def\E{\mathbf{E}}%
\global\long\def\P{\mathbf{P}}%
\global\long\def\op{\mathrm{op}}%
\global\long\def\var{\mathrm{Var}}%
\global\long\def\cov{\mathrm{Cov}}%
\global\long\def\ker{\mathrm{Ker}}%

\title{A Slightly Improved Bound for the KLS Constant}
\author{Arun Jambulapati\thanks{University of Washington, jmblpati@cs.washington.edu},
Yin Tat Lee\thanks{Microsoft Research and University of Washington, yintatlee@microsoft.com},
Santosh S. Vempala\thanks{Georgia Tech, vempala@gatech.edu} }
\maketitle
\begin{abstract}
We refine the recent breakthrough technique of Klartag and Lehec to
obtain an improved polylogarithmic bound for the KLS constant.
\end{abstract}

\section{Introduction}

The thin-shell constant, $\sigma_{n},$ and the KLS constant, $\psi_{n},$
are fundamental parameters of convex sets and log-concave densities
in $n$-dimensional Euclidean space. Roughly speaking, the thin-shell
constant is the width of an annulus that contains half the measure
of a distribution and the KLS constant is the reciprocal of the minimum
ratio, over subsets of measure at most half, of the surface measure
and measure of the subset. The thin-shell constant is bounded by the
KLS constant. The famous KLS conjecture posits that the KLS constant
is bounded by a universal constant independent of the dimension for
any isotropic logconcave density \cite{KLS95}. For background on
the conjecture and its myriad connections, see e.g., \cite{LV_survey2019}.
In a recent breakthrough paper, Klartag and Lehec established the
following bounds. 
\begin{thm}[\cite{klartag2022bourgain}]
$\sigma_{n}\lesssim\log^{4}n$, $\psi_{n}\lesssim\log^{5}n.$\footnote{Throughout this paper, we use $a\lesssim b$ to denote $a=O(b)$ and
$a\approx b$ to denote $a=\Theta(b)$.}
\end{thm}

This improved on previous bounds of Chen \cite{chen2021almost} and
Lee-Vempala \cite{lee2016eldan}. We give the following further improvement. 
\begin{thm}
$\sigma_{n}\lesssim\log^{2.2226}n$, $\psi_{n}\lesssim\log^{3.2226}n.$
\end{thm}

\subsection{Background}

In this section, we recap relevant definitions and results from \cite{Eldan2013,lee2016eldan,chen2021almost,klartag2022bourgain}.
Recall that a density $p$ in $\R^{n}$ is \emph{log-concave} if it
is nonnegative, has an integral of 1, and its logarithm is concave,
i.e., for any $x,y\in\R^{n},\lambda\in[0,1],$we have $p(\lambda x+(1-\lambda)y)\ge p(x)^{\lambda}p(y)^{1-\lambda}.$
A logconcave density $p$ is said to be $\alpha$-\emph{strongly}
log-concave or $\alpha$-\emph{uniformly} log-concave if the density
can be written as $p(x)=e^{-\alpha\norm x^{2}/2}q(x)$ where $q$
is a log-concave function. 
\begin{defn}[Constants]
For a distribution $\mu$ in $\Rn$, we define the following constants
\begin{itemize}
\item Thin-shell constant: $\sigma_{\mu}^{2}=\frac{\var_{\mu}(\|x\|^{2})}{n}$.
\item $\kappa_{\mu}^{2}=\sup_{\norm{\theta}=1}\norm{\E_{x\sim\mu}[\left\langle x,\theta\right\rangle xx^{\top}]}_{F}^{2}$ 
\item KLS constant: $\frac{1}{\psi_{\mu}}=\inf_{A\subseteq\Rn}\left\{ \frac{\int_{\partial A}d\mu}{\min(\mu(A),1-\mu(A))}\right\} $
where the infimum runs over all open sets $A\subseteq\R^{n}$ with
smooth boundary.
\end{itemize}
We define $\sigma_{n}$ be the supremum $\sigma_{\mu}$ over all \emph{isotropic
log-concave} distributions $\mu$ in $\Rn$. We define $\kappa_{n}$
and $\psi_{n}$ similarly.
\end{defn}

To relate the above constants, Eldan introduced the technique of stochastic
localization. 
\begin{defn}[Stochastic localization]
For any log-concave density $p$ in $\R^{n}$, we define the following
stochastic processes on log-concave densities $p_{t}$ with initial
density $p_{0}=p$:
\begin{itemize}
\item Eldan's stochastic localization \cite{Eldan2013}: $dp_{t}(x)=p_{t}(x)(x-a_{t})^{\top}A_{t}^{-\frac{1}{2}}dW_{t}$
where $a_{t}$ and $A_{t}$ are the mean and covariance of $p_{t}$
respectively.
\item Lee-Vempala variant \cite{lee2016eldan}: $dp_{t}(x)=p_{t}(x)(x-a_{t})^{\top}dW_{t}$.
\end{itemize}
For brevity, we refer to the first process as the Eldan process and
to the second as the LV process. Both processes are martingales in
the space of densities. 
\end{defn}

Eldan showed how to bound $\psi_{n}$ by controlling the operator
norm of $A_{t}$, the covariance of the density $p_{t}$ at time $t$.
In particular, he derived the following relation.
\begin{lem}[{\cite[Proposition 1.7, Lemma 1.6]{Eldan2013}}]
\label{lem:thin_shell_gap}
\[
\psi_{n}^{2}\lesssim\kappa_{n}^{2}\log n\lesssim\sigma_{n}^{2}\log^{2}n.
\]
\end{lem}

Following \cite{Eldan2013}, Lee and Vempala \cite{lee2016eldan}
suggested their variant. The benefit of the LV process is that $\|A_{t}\|_{\op}$
can be controlled by $\tr(A_{t}^{q})$, which can in turn be bounded
by inspection of its derivatives. Analysis of the LV process crucially
involves bounding the following object, which is one of the terms
in the derivative of $\tr(A_{t}^{q})$.
\begin{defn}
For any log-concave distribution $\mu$ on $\R^{n}$ with mean $a$
and for any set of symmetric matrices $A_{i}$, let 
\[
T_{\mu}(A_{1},A_{2},A_{3})=\E_{x,y\sim\mu}\prod_{i=1}^{3}(x-a)^{\top}A_{i}(y-a)
\]
where $x$ and $y$ are independent samples from $\mu$. We use the
same notation $T_{p}$ for any density $p$.
\end{defn}

By considering the case of $q=2$, i.e., $\tr(A_{t}^{2})$, \cite{lee2016eldan}
gave a bound of $\psi_{n}\lesssim n^{1/4}.$ Chen gave a substantially
better bound of $\log\psi_{n}\lesssim\sqrt{\log n\log\log n}$ by
choosing $q=\omega(1)$ and complementing \cite{lee2016eldan}'s analysis
with an improved bound on $T_{\mu}$ for all $t$-strongly log-concave
$\mu$:
\begin{lem}[{\cite[Lemma 11]{chen2021almost}}]
\label{lem:TAII_chen} For any $\alpha$-strongly log-concave $\mu$
with covariance $A$ and for any $q\geq3$, we have
\[
T_{\mu}(A^{q-2},I,I)\leq\frac{4}{\alpha}\tr A^{q}.
\]
\end{lem}

All previous works \cite{Eldan2013,lee2016eldan,chen2021almost} bound
$\psi_{n}$ via $\|A_{t}\|_{\op}$. The key idea of Klartag and Lehec
\cite{klartag2022bourgain} is to bound $\sigma_{n}^{2}$ via $\|a_{t}\|_{2}^{2}$
(where $a_{t}$ is the mean of $p_{t})$, and then bound $\psi_{n}$
via $\sigma_{n}$. The benefit is that $\|a_{t}\|_{2}^{2}$ involves
a 2-norm and thus possesses favorable smoothness properties compared
to $\tr(A_{t}^{q})$ considered in prior work. This leads to a polylogarithmic
bound on $\psi_{n}$. We summarize their technique below and discuss
its application in the next section.
\begin{lem}[{\cite[Eqn 45, 47]{klartag2022bourgain}}]
\label{lem:KL_key} For any log-concave density $p$ in $\Rn$, let
$p_{t}$ be the density given by the LV process starting at $p$.
Let $a_{t}$ be the mean of $p_{t}$. Then, we have
\[
\sigma_{p}^{2}\lesssim\int_{\lambda_{1}}^{\infty}\min_{t_{\lambda}>0}\left\{ \frac{1}{n\lambda^{2}}\E\|a_{t_{\lambda}}\|^{2}+\frac{1}{\lambda t_{\lambda}}\right\} d\lambda
\]
where $\lambda_{1}=\Omega(\psi_{p}^{-2})$ and the expectation is
taken over the LV process.
\end{lem}

\subsection{Discussion of the Klartag-Lehec bound}

\label{subsec:log4_discussion}

In this section, we explain how Lemma \ref{lem:KL_key} leads to the
result $\sigma_{n}\lesssim\log^{4}n$. The bound on $\psi_{n}$ follows
from Lemma \ref{lem:thin_shell_gap}. The intuition of this argument
is important for understanding our improvement. 

To apply Lemma \ref{lem:KL_key}, Klartag and Lehec start with an
arbitrary isotropic log-concave density $p$ and apply stochastic
localization, which gradually makes the distribution more and more
uniformly log-concave by creating a larger and larger Gaussian term
in the density. Roughly speaking, the idea then is that if the thin-shell
constant $\sigma_{p}$ is large, this implies that the squared length
of the mean is large at a small time. But the localization process
also proves that the squared length of the mean cannot grow quickly
--- for a small time interval, it in fact remains small. Comparing
these bounds allows us to conclude the thin-shell constant cannot
be too large. 

More precisely, we recall that
\[
\sigma_{p}^{2}\lesssim\int_{\lambda_{1}}^{\infty}\min_{t_{\lambda}>0}\left\{ \frac{1}{n\lambda^{2}}\E\|a_{t_{\lambda}}\|^{2}+\frac{1}{\lambda t_{\lambda}}\right\} d\lambda.
\]
To give some intuition of this integral, for some large enough constant
$C$, we set 
\[
t_{\lambda}\defeq\frac{C}{\sigma_{p}^{2}}\log^{2}\left(\frac{\lambda}{\lambda_{1}}\right)
\]
 and let $\lambda_{2}\defeq C\lambda_{1}\log^{2}n$ . Then we have
\begin{align*}
\sigma_{p}^{2} & \lesssim\int_{\lambda_{1}}^{\infty}\frac{\E\|a_{t_{\lambda}}\|^{2}}{n\lambda^{2}}d\lambda+\int_{\lambda_{1}}^{\infty}\frac{\sigma_{p}^{2}}{C\lambda\log^{2}(\lambda/\lambda_{1})}d\lambda\\
 & \lesssim\int_{\lambda_{1}}^{\lambda_{2}}\frac{\E\|a_{t_{\lambda}}\|^{2}}{n\lambda^{2}}d\lambda+\int_{\lambda_{2}}^{\infty}\frac{\E\|a_{t_{\lambda}}\|^{2}}{n\lambda^{2}}d\lambda+\int_{\lambda_{1}}^{\infty}\frac{\sigma_{p}^{2}}{C\lambda\log^{2}(\lambda/\lambda_{1})}d\lambda\\
 & \lesssim\int_{\lambda_{1}}^{\lambda_{2}}\frac{\E\|a_{t_{\lambda}}\|^{2}}{n\lambda^{2}}d\lambda+\frac{\sigma_{n}^{2}}{2},
\end{align*}
where the first line uses the definition of $t_{\lambda}$, the second
line splits the integral at $\lambda_{2}$, and the third line uses
the known fact that $\E\|a_{t}\|^{2}\le n$ for all $t$ \footnote{To see this, note that $a_{0}=0,\quad\frac{d}{dt}\E\|a_{t}\|^{2}=\tr A_{t}^{2}\ge0$,
the distribution of $a_{t}$ tends to $p_{0}$ as $t\rightarrow\infty$,
and $\E_{p_{0}}(\norm x^{2})=n$.} and the inequalities (for sufficiently large constant $C$)
\[
\int_{\lambda_{1}}^{\infty}\frac{1}{C\lambda\log^{2}(\lambda/\lambda_{1})}d\lambda\lesssim1\qquad\mbox{and \ensuremath{\qquad}\ensuremath{\int_{\lambda_{2}}^{\infty}\frac{\E\|a_{t_{\lambda}}\|^{2}}{n\lambda^{2}}}\ensuremath{d\ensuremath{\lambda\lesssim}}}\int_{\lambda_{2}}^{\infty}\frac{1}{\lambda^{2}}d\lambda=\frac{1}{\lambda_{2}}\lesssim\frac{\psi_{p}^{2}}{C\log^{2}n}\leq\frac{\sigma_{n}^{2}}{4}.
\]
 By picking the worst distribution $p$ such that $\sigma_{p}=\sigma_{n}$,
we have
\begin{equation}
\sigma_{n}^{2}\lesssim\int_{\lambda_{1}}^{\lambda_{2}}\frac{\E\|a_{t_{\lambda}}\|^{2}}{n\lambda^{2}}d\lambda.\label{eq:KL_key_apx}
\end{equation}
Note that the range of $\lambda$ in the integral is quite small:
it is between $\lambda_{1}$ and $O(\log^{2}n)\cdot\lambda_{1}$.
Hence, our choice of $t_{\lambda}$ is at most $C\sigma_{n}^{-2}(\log\log n)^{2}$.
To complete the analysis, we need to analyze $\E\|a_{t}\|^{2}$ for
$t$ near $\sigma_{n}^{-2}.$ By Itô's formula, one can calculate
that
\begin{equation}
\frac{d}{dt}\E\|a_{t}\|^{2}=\tr A_{t}^{2}\qquad\text{and}\qquad\frac{d}{dt}\E\tr A_{t}^{2}\leq T_{p_{t}}(I,I,I).\label{eq:dis_da_dtr}
\end{equation}
The key part of their analysis is to bound $T_{p}(I,I,I)$. 

Many quantities in the stochastic localization process (both Eldan
and LV versions) such as $a_{t}$, $A_{t}$ and $T_{p_{t}}(I,I,I)$
are well-behaved when $\|A_{t}\|_{\op}=O(1)$. Since $\|A_{t}\|_{\op}\leq2$
for $0\leq t\leq\frac{1}{c\kappa_{n}^{2}\log n}$ for some universal
constant $c$, one can show that $\tr A_{t}^{2}\lesssim n$ and $\E\|a_{t}\|^{2}\lesssim t\cdot n$
for $t$ in this range.

Beyond time $t_{1}\defeq\frac{1}{c\kappa_{n}^{2}\log n}$, Chen showed
that $\tr A_{t}^{q}$ can only grow polynomially as 
\[
\tr A_{t}^{q}\lesssim(\frac{t}{t_{1}})^{O(q^{2})}n
\]
using Lemma \ref{lem:TAII_chen}. Unfortunately, his proof only works
for $q\geq3$. Although we could indirectly bound $\frac{d}{dt}\E\|a_{t}\|^{2}$
by using the inequality $\tr A_{t}^{2}\leq(\tr I)^{1/3}(\tr A_{t}^{3})^{2/3}$,
as we will see, the degree of this polynomial growth directly affects
the exponent of the logarithm in the final bound. To get an improved
bound for $q=2$, Klartag and Lehec gave a very different proof. They
showed 
\begin{equation}
T_{\mu}(I,I,I)\leq\frac{\gamma}{\alpha}\tr A^{2}\text{ with }\gamma=3\label{eq:tensor_bound}
\end{equation}
for any $\alpha$-strongly log-concave distribution $\mu$ with covariance
$A$. Throughout this paper, we use $\gamma$ to denote the best bound
in (\ref{eq:tensor_bound}). We will later show that $\gamma\leq2\sqrt{2}$.

Using (\ref{eq:tensor_bound}) and (\ref{eq:dis_da_dtr}), 
\[
\E\|a_{t}\|^{2}\lesssim t\cdot\left(\frac{t}{t_{1}}\right)^{\gamma}\cdot n\text{ for }t\geq t_{1}.
\]
With this growth bound on $\E\|a_{t}\|^{2}$, using (\ref{eq:KL_key_apx})
we have, 
\[
\sigma_{n}^{2}\lesssim\int_{\lambda_{1}}^{\lambda_{2}}\frac{\E\|a_{t_{\lambda}}\|^{2}}{n\lambda^{2}}d\lambda\lesssim\int_{\lambda_{1}}^{\lambda_{2}}\frac{t_{\lambda}^{\gamma+1}}{t_{1}^{\gamma}\lambda^{2}}d\lambda\lesssim\int_{\lambda_{1}}^{\lambda_{2}}\frac{\sigma_{p}^{-2\gamma-2}\log^{2\gamma+2}(\lambda/\lambda_{1})}{t_{1}^{\gamma}\lambda^{2}}d\lambda\lesssim\sigma_{p}^{-2\gamma-2}\lambda_{1}^{-1}t_{1}^{-\gamma}.
\]
Noting $t_{1}\approx\frac{1}{\kappa_{n}^{2}\log n}$ and $\psi_{n}^{2}\lesssim\kappa_{n}^{2}\log n\lesssim\sigma_{n}^{2}\log^{2}n$,
we have 

\[
t_{1}^{-\gamma}\lesssim\kappa_{n}^{2\gamma}\log^{\gamma}n\lesssim\sigma_{n}^{2\gamma}\log^{2\gamma}n\qquad\text{and}\qquad\lambda_{1}^{-1}\lesssim\psi_{p}^{2}\leq\psi_{n}^{2}\lesssim\sigma_{n}^{2}\log^{2}n.
\]

Substituting the above and choosing constants appropriately yields
\[
\sigma_{n}\lesssim\log^{\gamma+1}n.
\]
Note that the exponent $\gamma+1$ above exactly depends on the parameter
$\gamma$ defined in (\ref{eq:tensor_bound}).

Moreover, if there is no gap between the KLS and thin-shell constants
(namely, $\psi_{n}\lesssim\kappa_{n}\lesssim\sigma_{n}$ rather than
$\psi_{n}\lesssim\kappa_{n}\sqrt{\log n}\lesssim\sigma_{n}\log n$),
then by essentially the same proof one can get an improved bound of
\[
\sigma_{n}\lesssim\log^{\frac{\gamma}{2}}n.
\]

\subsection{Our contributions}

\paragraph{Unconditional bound on $T_{p}(I,I,I)$.}

Our first contribution is to show $\gamma\leq2\sqrt{2}$ (Lemma \ref{lem:TIII}).
To do this, we first give a refined decomposition of strongly log-concave
densities via the Eldan process (Lemma \ref{lem:stoc_decomp}). With
this decomposition, we follow the argument of Klartag and Lehec (partially
inspired by \cite{jiang2020generalized}). We think this decomposition
may be of independent interest.

\paragraph{Conditional bound on $T_{p}(I,I,I)$.}

If $\gamma=o(1)$, then we could have an almost logarithmic bound
on $\sigma_{n}$. For one-dimensional log-concave densities, the standard
localization lemma together with a computer search suggests that $\gamma\simeq0.37$
(see Appendix \ref{sec:gamma}). Our second contribution is to show
$\gamma=o(1)$ whenever the covariance matrix $A$ is sufficiently
``spiked''. More formally, Lemma \ref{lem:TIII_conditional} proves
the following implication: 
\[
\tr A^{3}=\frac{o(1)}{\alpha}\tr A^{2}\implies T_{p}(I,I,I)\leq\frac{o(1)}{\alpha}\tr A^{2}
\]
Curiously, its proof requires the use of the LV process and does not
work with the Eldan process.

Next, we prove that the above condition on $\tr A^{3}$ is satisfied
for (small) positive time beyond the threshold $t_{1}$. Recall that
$t_{1}\approx(\kappa_{n}^{2}\log n)^{-1}$ while $t_{\lambda}\approx\sigma_{n}^{-2}\lesssim\kappa_{n}^{-2}\log n$.
Our new bound works well from $t_{1}$ to $t_{1.5}\approx\kappa_{n}^{-2}$
and essentially shows that $\|a_{t}\|^{2}$ does not grow much in
that time period. Beyond this point, we can use the unconditional
bound above. Therefore, it effectively ``halves'' the dependence
on $\gamma$. The main lemma of this part (Lemma \ref{lem:mu2_late})
gives an improved growth bound on the norm of $\norm{a_{t}}^{2}$
up to a time slightly beyond $t_{1.5}$ to a time we call $t^{*}$.
Our final bound is $\sigma_{n}\lesssim\log^{2.2226}n$ (Theorem \ref{thm:main_result}).

\paragraph{Sub-logarithmic bound assuming $\psi_{n}\approx\sigma_{n}$.}

As discussed in Section \ref{subsec:log4_discussion}, \emph{if} there
is ``no gap'' between KLS and thin-shell, i.e., the KLS constant
is within a universal constant factor of the thin-shel constant, then
the Klartag-Lehec bound becomes $\sigma_{n}\lesssim\log^{\frac{\gamma}{2}}n$.
With our analysis above, we can instead show that ``no gap'' implies
$\sigma_{n}\lesssim\log^{0.6476}n$, a sub-logarithmic bound, even
for the current $\gamma=2\sqrt{2}$ (Theorem \ref{thm:no_gap}). 

\subsection{Preliminaries}

In this section, we list various facts and results we use in this
paper.
\begin{lem}[{\cite[Lemma 2]{cordero2004b}}]
\label{lem:half_poincare}For any $\alpha$-strongly log-concave
distribution $\mu$ and any smooth function $f\in L_{2}(\mu)$ such
that $\int fd\mu=\int\nabla fd\mu=0$, one has
\[
\int f^{2}d\mu\leq\frac{1}{2\alpha}\int\|\nabla f\|^{2}d\mu.
\]
\end{lem}

\begin{lem}[Properties of the Eldan process]
\label{lem:Eldan_property}Let $p_{t}$ given by Eldan's process.
Let $a_{t}$ and $A_{t}$ be its mean and covariance. Then, we have
\begin{itemize}
\item $da_{t}=A_{t}^{1/2}dW_{t}$ \cite[Eqn 16]{Eldan2013}.
\item $dA_{t}=\E_{x\sim p_{t}}(x-a_{t})(x-a_{t})^{\top}(x-a_{t})^{\top}A_{t}^{-1/2}dW_{t}-A_{t}dt$
\cite[Page 16]{Eldan2013}.
\item If the starting distribution $p_{0}$ is $1$-strongly logconcave,
then $\|A_{t}\|_{\op}\leq e^{-t}$ \cite[Lemma 6]{eldan2014bounding}.
\end{itemize}
\end{lem}

\begin{lem}[Properties of the LV process]
\label{lem:LV_property}Let $p_{t}$ given by the LV process. Let
$a_{t}$ and $A_{t}$ be its mean and covariance. Then, we have
\begin{itemize}
\item $da_{t}=A_{t}dW_{t}$ \cite[Page 12]{lee2016eldan}.
\item $dA_{t}=\E_{x\sim p_{t}}(x-a_{t})(x-a_{t})^{\top}(x-a_{t})^{\top}dW_{t}-A_{t}^{2}dt$
\cite[Lemma 28]{lee2016eldan}.
\item $p_{t}(x)\propto e^{c_{t}^{\top}x-\frac{t}{2}\|x\|_{2}^{2}}p_{0}(x)$
\cite[Definition 26]{lee2016eldan}.
\end{itemize}
\end{lem}

The next lemma summarizes properties of $T$; the equivalent isotropic
versions are stated in \cite{jiang2020generalized}.
\begin{lem}[Properties of $T_{\mu}$]
\label{lem:T_property}For any log-concave distribution $\mu$ with
covariance $A$, and any symmetric matrices $M$ and $N$, we have
\begin{itemize}
\item $T_{\mu}(M_{1},M_{2},M_{3})\geq0$ for any positive definite matrices
$M_{1},M_{2},M_{3}$ \cite[Lemma 39]{jiang2020generalized}.
\item $T_{\mu}(M,I,I)\lesssim\kappa_{n}^{2}\cdot\tr|A^{1/2}MA^{1/2}|\cdot\|A\|_{\op}^{2}$
\cite[Lemma 40]{jiang2020generalized}.
\item $T_{\mu}(M,N,I)\leq T_{\mu}(|M|^{s},I,I)^{1/s}\cdot T_{\mu}(|N|^{t},I,I)^{1/t}$
for $s,t\ge1$ with $(1/s)+(1/t)=1$ \cite[Lemma 40]{jiang2020generalized}.
\item $T_{\mu}(N^{1/2}M^{\alpha}N^{1/2},N^{1/2}M^{1-\alpha}N^{1/2},C)\le T_{\mu}(N^{1/2}MN^{1/2},N,C)$
for any PSD matrices $M.N,C$ and $\alpha\in[0,1]$ \cite[Lemma 41]{jiang2020generalized}. 
\end{itemize}
\end{lem}

\begin{lem}[{Growth of the LV process \cite[Lemma 5.2]{klartag2022bourgain}}]
\label{lem:LV_growth}Let $p_{t}$ given by the LV process with an
initial isotropic log-concave distribution $p_{0}$. Let $A_{t}$
be its covariance. For $T\leq\frac{1}{c\kappa_{n}^{2}\log n}$ with
large enough constant $c$, we have
\[
\P\left[\|A_{t}\|_{\op}\geq2\text{ for }0\leq t\leq T\right]\leq\exp(-1/(cT)).
\]
\end{lem}

\section{Bounding $T_{\mu}(I,I,I)$}

In this section, we have two ways to bound $T_{p}(I,I,I)$. The first
way is unconditional and relies on Eldan's process. The second way
is conditional and relies on the LV process. Our overall proof will
effectively use both variants of localization. 

\subsection{Unconditional Bound}

By Eldan's process, we have the following decomposition of $1$-strongly
log-concave distributions. This lemma is a slightly strengthened version
of \cite[Lemma 4.1]{klartag2022bourgain} with an extra conclusion
$\E Q_{t}^{2}=e^{-t}A$. The proof follows from the remark at the
end of the proof of \cite[Lemma 4.1]{klartag2022bourgain}.
\begin{lem}
\label{lem:stoc_decomp}For any 1-strongly log-concave distribution
$\mu$ with mean $0$, we have 
\[
\mu\sim\int_{0}^{\infty}Q_{t}dW_{t}
\]
for positive definite matrices $Q_{t}$ with $0\preceq Q_{t}\preceq e^{-t/2}I$.
Moreover, we have
\[
\E Q_{t}^{2}=e^{-t}A
\]
where $A$ is the covariance of $\mu$ and the expectation is over
the stochastic process generating $Q_{t}$.
\end{lem}

\begin{proof}
We define $p_{t}$ according to the Eldan process with initial density
given by $d\mu$. For $a_{t}$, the mean of $p_{t}$, Lemma \ref{lem:Eldan_property}
shows
\[
da_{t}=A_{t}^{1/2}dW_{t}\qquad\text{and}\qquad\|A_{t}\|_{\op}\leq e^{-t}
\]
where $A_{t}$ is the covariance of the $p_{t}$. Since $\E p_{t}=p_{0}$
(by definition) and that $p_{\infty}$ is a delta measure at $a_{\infty}$
(because $\|A_{\infty}\|_{\op}=0$), we have
\[
\mu\sim a_{\infty}=\int_{0}^{\infty}A_{t}^{1/2}dW_{t}.
\]
This gives the first part with $Q_{t}=A_{t}^{1/2}$.

For the second part, Lemma \ref{lem:Eldan_property} shows that
\[
dA_{t}=\E_{x\sim p_{t}}(x-a_{t})(x-a_{t})^{\top}(x-a_{t})^{\top}A_{t}^{-1/2}dW_{t}-A_{t}dt
\]
Taking expectations, we have $\frac{d}{dt}\E A_{t}=-\E A_{t}$. Solving
it, we have $\E A_{t}=e^{-t}A.$
\end{proof}
We can use the above decomposition to bound $T_{\mu}(I,I,I)$ for
$\alpha$-strongly log-concave distributions. The proof closely follows
that of \cite[Lemma 4.2]{klartag2022bourgain}, except that we avoid
one Cauchy-Schwarz, and use the stronger bound on the operator norm
of $Q_{t}$ in the decomposition above.
\begin{lem}
\label{lem:TIII}For any $\alpha$-strongly log-concave distribution
$\mu$ with covariance $A$, we have
\[
T_{\mu}(I,I,I)\leq\frac{2\sqrt{2}}{\alpha}\tr A^{2}.
\]
\end{lem}

\begin{proof}
By scaling $\mu$ and shifting, we can assume $\alpha=1$ and $\mu$
has mean $0$. Lemma \ref{lem:stoc_decomp} shows that $\mu\sim\int_{0}^{\infty}Q_{s}dW_{s}$.
For any $x\sim\mu$, we define the random path $x_{t}$ by $x_{t}=\int_{0}^{t}Q_{s}ds$:
note that $x_{\infty}=x$. By Itô's formula, we have
\begin{align*}
T_{\mu}(I,I,I)=\E_{x,y\sim\mu}(x^{\top}y)^{3} & =3\int_{0}^{\infty}\E_{x,y}x_{t}^{\top}y\cdot|Q_{t}y|^{2}dt\\
 & \leq3\int_{0}^{\infty}\E_{x}\sqrt{\E_{y}(x_{t}^{\top}y)^{2}\cdot\var|Q_{t}y|^{2}}dt.
\end{align*}
The inequality follows from $\E_{y}x_{t}^{\top}y\cdot|Q_{t}y|^{2}=\E_{y}x_{t}^{\top}y\cdot\left(|Q_{t}y|^{2}-\text{\ensuremath{\left(\E|Q_{t}y|^{2}\right)}}^{2}\right)$
(since $\E[y]$= 0) and the Cauchy-Schwarz inequality.

Next, since $\mu$ is $1$-strongly log-concave and the gradient of
$|Q_{t}y|^{2}$ has mean $0$, Lemma \ref{lem:half_poincare} yields
\[
\var|Q_{t}y|^{2}\leq\frac{1}{2}\E|2Q_{t}^{2}y|^{2}=2\tr Q_{t}^{4}A.
\]
Hence, we have
\begin{align*}
T_{\mu}(I,I,I) & \leq3\sqrt{2}\int_{0}^{\infty}\E_{x}\sqrt{\E_{y}(x_{t}^{\top}y)^{2}\cdot\tr Q_{t}^{4}A}dt\\
 & \substack{\leq}
3\sqrt{2}\int_{0}^{\infty}e^{-t/2}\E_{x}\sqrt{x_{t}^{\top}Ax_{t}\cdot\tr Q_{t}^{2}A}dt
\end{align*}
where we used $Q_{t}\preceq e^{-t/2}I$ and that $y$ has mean $0$
and covariance $A$. By Lemma \ref{lem:stoc_decomp}, we have $\E Q_{s}^{2}=e^{-t}A$
and
\[
\E x_{t}x_{t}^{\top}=\int_{0}^{t}\E Q_{s}^{2}ds=(1-e^{-t})A.
\]
By another application of Cauchy-Schwarz, we have
\begin{align*}
T_{\mu}(I,I,I) & \leq3\sqrt{2}\int_{0}^{\infty}e^{-t/2}\cdot\sqrt{\tr\E x_{t}x_{t}^{\top}A\cdot\tr\E Q_{s}^{2}A}dt\\
 & =3\sqrt{2}\int_{0}^{\infty}e^{-t/2}\cdot\sqrt{(1-e^{-t})\cdot e^{-t}}dt\cdot\tr A^{2}\\
 & =2\sqrt{2}\cdot\tr A^{2}.
\end{align*}
\end{proof}

\subsection{Conditional Bound}

Our conditional bound also follows the same framework as the previous
bound. We first give a stochastic decomposition of $\mu$ based on
the LV process. Then we use it to bound $T_{\mu}(I,I,I)$. 
\begin{lem}
\label{lem:stoc_decomp_2}Let $\mu$ be an $\alpha$-strongly log-concave
distribution with mean $0$. Then, 
\[
\mu\sim\int_{0}^{\infty}Q_{t}dW_{t}
\]
for positive definite matrices $Q_{t}$ with $0\preceq Q_{t}\preceq\frac{1}{\alpha+t}I$.
Furthermore, we have
\[
\E Q_{t}\preceq A
\]
where $A$ is the covariance of $\mu$ and the expectation is over
the stochastic process generating $Q_{t}$.
\end{lem}

\begin{proof}
We follow the proof of Lemma \ref{lem:stoc_decomp}, but define $p_{t}$
in terms of the LV process. In this case, Lemma \ref{lem:LV_property}
shows that $p_{t}(x)\propto e^{c_{t}^{\top}x-\frac{t}{2}\|x\|_{2}^{2}}p_{0}(x)$
and that
\[
da_{t}=A_{t}dW_{t}
\]
where $A_{t}$ is the covariance of $p_{t}$. Since $p_{0}$ is $\alpha$-strongly
log-concave, we have that $p_{t}$ is $(\alpha+t)$-strongly log-concave,
and therefore
\[
\|A_{t}\|_{\op}\leq\frac{1}{\alpha+t}.
\]
This gives the first part with $Q_{t}=A_{t}$. For the second part,
Lemma \ref{lem:LV_property} shows that
\[
dA_{t}=\E_{x\sim p_{t}}(x-a_{t})(x-a_{t})^{\top}(x-a_{t})^{\top}dW_{t}-A_{t}^{2}dt
\]
Taking expectations, we have $\frac{d}{dt}\E A_{t}\preceq0$. Hence,
we have $\E A_{t}\preceq A$.
\end{proof}
This lemma can be used in place of Lemma \ref{lem:stoc_decomp} in
the proof of Lemma \ref{lem:TIII}. However, doing so leads to a worse
constant because this lemma does not capture the fact that $\E Q_{t}$
is decreasing with $t$. The benefit of this decomposition is that
$\tr Q_{t}^{q}$ can be controlled by $\tr A^{q}$ throughout the
process, as we prove below.
\begin{lem}
\label{lem:Aq_growth}For any $\alpha$-strongly log-concave distribution
$\mu$ with covariance $A$, let $p_{t}$ be the result of the LV
process with initial density $p_{0}$ given by $\mu$. Let $A_{t}$
be the covariance of $p_{t}$. For any $q\geq3$ we have
\[
\E\tr A_{t}^{q}\leq\left(1+\frac{t}{\alpha}\right)^{q(q-1)}\tr A^{q}.
\]
In particular, the $Q_{t}$ defined in Lemma \ref{lem:stoc_decomp_2}
satisfies $\E\tr Q_{t}^{q}\leq(1+\frac{t}{\alpha})^{q(q-1)}\tr A^{q}$.
\end{lem}

\begin{proof}
Let $a_{t}$ be the mean of $p_{t}$. Lemma \ref{lem:trQ_hess} with
$b=0$ shows that
\begin{align*}
\frac{d}{dt}\E\tr A_{t}^{q} & \leq\frac{q(q-1)}{2}T_{p_{t}}(A_{t}^{q-2},I,I).
\end{align*}
Since $p_{t}$ is $(\alpha+t)$-strongly log-concave (Lemma \ref{lem:LV_property}),
we can use Lemma \ref{lem:TAII} to bound the last term to get
\[
\frac{d}{dt}\E\tr A_{t}^{q}\leq\frac{q(q-1)}{\alpha+t}\tr A_{t}^{q}
\]
Solving this equation yields the first result. The second result follows
from the fact that $Q_{t}$ in Lemma \ref{lem:stoc_decomp_2} is exactly
$A_{t}$.
\end{proof}
Using the new stochastic decomposition (Lemma \ref{lem:stoc_decomp_2})
and the new property about $\tr Q_{t}^{q}$ (Lemma \ref{lem:Aq_growth}),
we have an improved bound on $T_{\mu}(I,I,I)$.
\begin{lem}
\label{lem:TIII_conditional}Let $\mu$ be an $\alpha$-strongly log-concave
distribution with covariance $A$. For $q$ such that $8\geq q\geq3$,
assume that $\tr A^{q}\leq\frac{1}{\alpha^{(q-2)}\zeta}\tr A^{2}$.
Then, 
\[
T_{\mu}(I,I,I)\leq\frac{12}{\alpha\zeta^{1/(2(q^{2}-2))}}\tr A^{2}.
\]
Alternatively, 
\[
T_{\mu}(I,I,I)\leq\frac{12}{\alpha^{3}}\cdot\left(\alpha^{q}\tr A^{q}\right)^{c}\left(\alpha^{2}\tr A^{2}\right)^{1-c}
\]
with $c=\frac{1}{2(q^{2}-2)}$.
\end{lem}

\begin{proof}
By centering, we can assume $\mu$ has mean $0$. Lemma \ref{lem:stoc_decomp_2}
shows that $\mu\sim\int_{0}^{\infty}Q_{s}dW_{s}$. For $x\sim\mu$,
we define the random path $x_{t}$ by $x_{t}=\int_{0}^{t}Q_{s}ds$.
Note that $x_{\infty}=x$. By Itô's formula and the Cauchy-Schwarz
inequality, we have
\begin{align}
T_{\mu}(I,I,I)=\E_{x,y\sim\mu}(x^{\top}y)^{3} & =3\int_{0}^{\infty}\E_{x,y}x_{t}^{\top}y\cdot|Q_{t}y|^{2}dt\nonumber \\
 & \leq3\int_{0}^{\infty}\E_{x}\sqrt{\E_{y}(x_{t}^{\top}y)^{2}\cdot\var|Q_{t}y|^{2}}dt.\label{eq:TIII_cond_1}
\end{align}
Since $\mu$ is $\alpha$-strongly log-concave and the gradient of
$|Q_{t}y|^{2}$ has mean $0$, Lemma \ref{lem:half_poincare} shows
that
\begin{equation}
\var|Q_{t}y|^{2}\leq\frac{1}{2\alpha}\E|2Q_{t}^{2}y|^{2}=\frac{2}{\alpha}\tr Q_{t}^{4}A.\label{eq:TIII_cond_2}
\end{equation}
Up to this point, the proof is the same as that of Lemma \ref{lem:TIII}. 

Since $x_{t}$ is a martingale, we have
\[
A=\E_{x}x_{\infty}x_{\infty}^{\top}=\E_{x}x_{t}x_{t}^{\top}+\E_{x}(x_{\infty}-x_{t})(x_{\infty}-x_{t})^{\top}.
\]
In particular, this shows that $\E_{x}x_{t}x_{t}^{\top}\preceq A$.
Using this, (\ref{eq:TIII_cond_1}) and (\ref{eq:TIII_cond_2}), we
have
\begin{align}
T_{\mu}(I,I,I) & \leq3\sqrt{\frac{2}{\alpha}}\int_{0}^{\infty}\sqrt{\tr A^{2}\cdot\tr Q_{t}^{4}A}dt.\label{eq:tensor_bound-1}
\end{align}
Now we split the proof into two cases for bounding the term $\E\tr Q_{t}^{4}A$. 

In the first case, Lemma \ref{lem:stoc_decomp_2} shows that $Q_{t}\preceq\frac{1}{\alpha+t}$
and $\E Q_{t}\preceq A$. Therefore, we have
\[
\E\tr Q_{t}^{4}A\leq\frac{1}{(\alpha+t)^{3}}\E\tr Q_{t}A\leq\frac{1}{(\alpha+t)^{3}}\tr A^{2}.
\]

In the second case, Lemma \ref{lem:Aq_growth} together with the assumption
of the current lemma shows that 
\[
\tr Q_{t}^{q}\leq\left(1+\frac{t}{\alpha}\right)^{q(q-1)}\tr A^{q}\leq\frac{\alpha^{-(q-2)}}{\zeta}\left(1+\frac{t}{\alpha}\right)^{q(q-1)}\tr A^{2}.
\]
Therefore, for $q\le8$, we have
\[
\E\tr Q_{t}^{4}A\quad\substack{(i)\\
{\leq}
}
\quad\frac{\E\tr Q_{t}^{q/2}A}{(\alpha+t)^{4-q/2}}\quad\substack{(ii)\\
{\leq}
}
\quad\frac{\sqrt{\E\tr Q_{t}^{q}\cdot\tr A^{2}}}{(\alpha+t)^{4-q/2}}\quad\substack{(iii)\\
{\leq}
}
\quad\frac{\alpha^{1-q/2}(1+\frac{t}{\alpha})^{q(q-1)/2}}{(\alpha+t)^{4-q/2}\zeta^{1/2}}\tr A^{2}.
\]
Inequality $(i)$ used $Q_{t}\preceq\frac{1}{\alpha+t}I$, $(ii)$
used the Cauchy-Schwarz inequality, and $(iii)$ follows from Lemma
\ref{lem:Aq_growth}.

Observe that the above two bounds are equal at $t=\alpha s^{*}$,
where $s^{*}=\zeta^{\frac{1}{q^{2}-2}}-1$. Thus we have the bound
\[
\E\tr Q_{t}^{4}A\leq\begin{cases}
\frac{1}{(\alpha+t)^{3}}\tr A^{2} & \text{if }\frac{t}{\alpha}\geq s^{*}\\
\frac{\alpha^{1-q/2}(1+\frac{t}{\alpha})^{q(q-1)/2}}{(\alpha+t)^{4-q/2}\zeta^{1/2}}\tr A^{2} & \text{otherwise.}
\end{cases}
\]
Substituting this into \ref{eq:tensor_bound-1} and splitting the
integral yields
\begin{align*}
\frac{T_{\mu}(I,I,I)}{3\sqrt{2}\tr A^{2}}\leq & \frac{1}{\sqrt{\alpha}}\int_{0}^{\alpha s^{*}}\frac{\alpha^{1/2-q/4}(1+\frac{t}{\alpha})^{q(q-1)/4}}{(\alpha+t)^{2-q/4}\zeta^{1/4}}dt+\frac{1}{\sqrt{\alpha}}\int_{\alpha s^{*}}^{\infty}\frac{1}{(\alpha+t)^{3/2}}dt.
\end{align*}
Substituting $u=\frac{t}{\alpha}$, we have
\begin{align*}
\frac{T_{\mu}(I,I,I)}{3\sqrt{2}\tr A^{2}}\leq & \frac{1}{\alpha\zeta^{1/4}}\int_{0}^{s^{*}}(1+u)^{\frac{q^{2}-8}{4}}du+\frac{1}{\alpha}\int_{s^{*}}^{\infty}\frac{1}{(1+u)^{3/2}}du\\
= & \frac{4}{(q^{2}-4)\alpha\zeta^{1/4}}\left((1+s^{*})^{\frac{q^{2}-4}{4}}-1\right)+\frac{2}{\alpha\sqrt{1+s^{*}}}\\
= & \frac{4}{(q^{2}-4)\alpha\zeta^{1/4}}\left(\zeta^{\frac{q^{2}-4}{4(q^{2}-2)}}-1\right)+\frac{2}{\alpha\zeta^{\frac{1}{2(q^{2}-2)}}}.
\end{align*}
Further simplifying, we have
\[
\frac{\alpha T_{\mu}(I,I,I)}{3\sqrt{2}\tr A^{2}}\leq\frac{4}{(q^{2}-4)\zeta^{\frac{1}{2(q^{2}-2)}}}+\frac{2}{\zeta^{\frac{1}{2(q^{2}-2)}}}\leq\frac{2.8}{\zeta^{1/(2(q^{2}-2))}}.
\]
This gives the first claim. The second claim follows by setting $\zeta=\alpha^{-(q-2)}\frac{\tr A^{2}}{\tr A^{q}}$.
\end{proof}

\section{Bounding $\protect\tr A_{t}^{q}$}

In this section, we derive refined bounds on $\tr A_{t}^{q}$. We
will use these to apply our conditional bounds on $T_{\mu}(I,I,I)$.
We first bound $\tr A_{t}^{q}$ when $t$ is small. Let $A^{+}$ denote
the restriction of a symmetric matrix $A$ to the span of its positive
eigenvectors. 
\begin{lem}
\label{lem:trAq}Let $p_{t}$ be given by the LV process with an initial
isotropic log-concave density $p_{0}$. Assume that $p_{0}$ is supported
on a ball with radius $n$. Let $A_{t}$ be the covariance of $p_{t}$,
$\overline{A}_{t}=(A_{t}-I)^{+}$, and $t_{1}=\frac{1}{c\kappa_{n}^{2}\log n}$
for some large enough constant $c$ depending only on $q$. Then,
for any $0\leq t\leq t_{1}$ and any $q\geq3$, we have
\[
\E\tr\overline{A}_{t}^{q}\leq1+(Cq\kappa_{n}^{2}t)^{q/2}\cdot n
\]
for some universal constant $C>0$.
\end{lem}

\begin{proof}
Let $\Phi_{t}=\tr\overline{A}_{t}^{q}$, and $E_{t}$ be the event
that $\|A_{s}\|_{\op}\leq2\text{ for all }0\leq s\leq t$. To bound
the trace for the case $\|A_{s}\|$ is small, we define
\[
\Psi_{t}=\E[\Phi_{t}\cdot1_{E_{t}}].
\]
Lemma \ref{lem:trQ_hess} (with $b=1$) shows that 
\begin{align}
\frac{d}{dt}\E\Phi_{t}\leq & \frac{q(q-1)}{2}T_{p_{t}}(\overline{A}_{t}^{q-2},I,I)\label{eq:d_trAq_Phi}
\end{align}
where the expectation is over all randomness before time $t$. Lemma
\ref{lem:T_property} followed by Cauchy-Schwarz shows that
\[
T_{p_{t}}(\overline{A}_{t}^{q-2},I,I)\lesssim\kappa_{n}^{2}\cdot\tr\overline{A}_{t}^{q-2}A_{t}\cdot\|A_{t}\|_{\op}^{2}\leq\kappa_{n}^{2}\cdot\tr\overline{A}_{t}^{q-2}\cdot\|A_{t}\|_{\op}^{3}\lesssim\kappa_{n}^{2}\cdot\Phi_{t}^{1-\frac{2}{q}}n^{\frac{2}{q}}\cdot\|A_{t}\|_{\op}^{3}.
\]
Since $f(x)=x^{1-\frac{2}{q}}$ is concave, Jensen's inequality implies
\[
\frac{d}{dt}\Psi_{t}\lesssim q^{2}\kappa_{n}^{2}\cdot\E[\Phi_{t}^{1-\frac{2}{q}}1_{E_{t}}]n^{\frac{2}{q}}\lesssim q^{2}\kappa_{n}^{2}\Psi_{t}^{1-\frac{2}{q}}n^{\frac{2}{q}}.
\]
Hence, we have $\frac{d}{dt}\Psi_{t}^{2/q}\lesssim q\kappa_{n}^{2}n^{\frac{2}{q}}.$
Since $\Psi_{0}=0$, we have
\[
\Psi_{t}\leq(Cq\kappa_{n}^{2}n^{\frac{2}{q}}t)^{q/2}
\]
for some universal constant $C$.

Next we bound $\E\Phi_{t}$. By Lemma \ref{lem:LV_growth}, with $t_{1}\approx(q\kappa_{n}^{2}\log n)^{-1}$,
we have
\[
\P\left[\|A_{t}\|_{\op}\geq2\text{ for }0\leq t\leq t_{1}\right]\leq\frac{1}{n^{2q+1}}.
\]
Since $p_{0}$ is supported on a ball with radius $n$, we have $\|A_{t}\|_{\op}\leq n^{2}$.
Hence, we have
\[
\E\tr\overline{A}_{t}^{q}\leq\Psi_{t}+\P(E_{t}^{c})n^{2q}\cdot n\leq(Cq\kappa_{n}^{2}t)^{q/2}\cdot n+1.
\]
This finishes the proof.
\end{proof}
\begin{lem}
\label{lem:trAq_late}Under the setting of Lemma \ref{lem:trAq} with
$3\leq q\leq4$, for any $t$ with $t_{1}\leq t\leq\hat{t}\defeq t_{1}\log^{\frac{1}{2q-2}}n$
we have $\E\tr A_{t}^{q}\lesssim n.$ Further, for any $t\geq\hat{t}$
we have
\[
\E\tr A_{t}^{q}\lesssim\left(\frac{t}{t_{1}}\right)^{q(q-1)}\frac{n}{\log^{q/2}n}
\]

where $t_{1}=\frac{1}{c\kappa_{n}^{2}\log n}$ for some constant $c$.
\end{lem}

\begin{proof}
Recall from (\ref{eq:d_trAq_Phi}) that Itô's formula yields 
\begin{align}
\frac{d}{dt}\E\tr\overline{A}_{t}^{q}\leq & \frac{q(q-1)}{2}T_{p_{t}}(\overline{A}_{t}^{q-2},I,I)\label{eq:d_trAq_late_Phi}
\end{align}
with $\Phi_{t}=\tr\overline{A}_{t}^{q}$. Now, we bound $T_{p_{t}}(\overline{A}_{t}^{q-2},I,I)$
using Lemma \ref{lem:TAII}. Let $p_{t}(M)$ be the density of the
distribution $M^{1/2}A_{t}^{-1/2}(x-a_{t})$ where $x\sim p_{t}$.
Note that $p_{t,M}$ has mean $0$ and covariance $M$. Using $A_{t}\preceq\overline{A_{t}}+I$
and Lemma \ref{lem:T_property}, we have
\begin{align*}
T_{p_{t}}(\overline{A}_{t}^{q-2},I,I) & =T_{p_{t}(I)}(A_{t}\overline{A}_{t}^{q-2},A_{t},A_{t})\\
 & \leq T_{p_{t}(I)}(\overline{A_{t}}^{q-1}+\overline{A_{t}}^{q-2},\overline{A_{t}}+I,\overline{A_{t}}+I)\\
 & =T_{p_{t}(I)}(\overline{A_{t}}^{q-1},\overline{A_{t}},\overline{A_{t}})+2T_{p_{t}(I)}(\overline{A_{t}}^{q-1},\overline{A_{t}},I)+T_{p_{t}(I)}(\overline{A_{t}}^{q-1},I,I)\\
 & \qquad+T_{p_{t}(I)}(\overline{A_{t}}^{q-2},\overline{A_{t}},\overline{A_{t}})+2T_{p_{t}(I)}(\overline{A_{t}}^{q-2},\overline{A_{t}},I)+T_{p_{t}(I)}(\overline{A_{t}}^{q-2},I,I)\\
 & \le T_{p_{t}(I)}(\overline{A_{t}}^{q-1},\overline{A_{t}},\overline{A_{t}})+3T_{p_{t}(I)}(\overline{A_{t}}^{q},I,I)+3T_{p_{t}(I)}(\overline{A_{t}}^{q-1},\overline{A_{t}},I)+T_{p_{t}(I)}(\overline{A_{t}}^{q-2},I,I).
\end{align*}
For the first term, we note that $p_{t,\overline{A_{t}}}$ is $t$-strongly
log-concave because $p_{t}$ is $t$-strongly log-concave and that
$\overline{A_{t}}\preceq A_{t}$. Hence, Lemma \ref{lem:TAII} shows
that
\[
T_{p_{t}(I)}(\overline{A_{t}}^{q-1},\overline{A_{t}},\overline{A_{t}})=T_{p_{t}(\overline{A_{t}})}(\overline{A_{t}}^{q-2},I,I)\leq\frac{2}{t}\tr\overline{A_{t}}^{q}.
\]
For other terms, we use Lemma \ref{lem:T_property} and Hölder's inequality
to get 
\[
T_{p_{t}(I)}(\overline{A_{t}}^{q-1},\overline{A_{t}},I)\lesssim T_{p_{t}(I)}(\overline{A_{t}}^{q},I,I).
\]
Next, as Lemma \ref{lem:T_property} gives $T_{p_{t},I}(\overline{A_{t}}^{k},I,I)\lesssim\kappa_{n}^{2}\tr\overline{A_{t}}^{k}$
we obtain
\begin{align*}
T_{p_{t}}(|A_{t}-I|^{q-2},I,I) & \leq\frac{2}{t}\tr\overline{A_{t}}^{q}+O(\kappa_{n}^{2})(\tr\overline{A_{t}}^{q}+\tr\overline{A_{t}}^{q-2})\\
 & \leq\left(\frac{2}{t}+O(\kappa_{n}^{2})\right)\tr\overline{A_{t}}^{q}+O(\kappa_{n}^{2})\cdot\left(\tr\overline{A_{t}}^{q}\right)^{1-\frac{2}{q}}n^{\frac{2}{q}}.
\end{align*}
Substituting into Equation\ref{eq:d_trAq_late_Phi} gives
\begin{align*}
\frac{d}{dt}\E\Phi_{t}\leq & \left(\frac{q(q-1)}{t}+C\kappa_{n}^{2}\right)\E\Phi_{t}+C\kappa_{n}^{2}\left(\E\Phi_{t}\right)^{1-\frac{2}{q}}n^{\frac{2}{q}}
\end{align*}
for some universal constant $C>0$. To solve this inequality, we let
$\alpha=q(q-1)$, $\beta=C\kappa_{n}^{2}$ and define
\[
\Psi_{t}=\left(\frac{t}{t_{1}}\right)^{-\alpha}e^{-\beta(t-t_{1})}\E\Phi_{t}.
\]
Under these substitutions our equation gives
\begin{align*}
\frac{d}{dt}\Psi_{t} & =-\left(\frac{\alpha}{t}+\beta\right)\Psi_{t}+\left(\frac{t}{t_{1}}\right)^{-\alpha}e^{-\beta(t-t_{1})}\frac{d}{dt}\E\Phi_{t}\\
 & \leq-\left(\frac{\alpha}{t}+\beta\right)\Psi_{t}+\left(\frac{t}{t_{1}}\right)^{-\alpha}e^{-\beta(t-t_{1})}\left((\frac{\alpha}{t}+\beta)\E\Phi_{t}+\beta(\E\Phi_{t})^{1-\frac{2}{q}}n^{\frac{2}{q}}\right)\\
 & =\beta\left(\frac{t}{t_{1}}\right)^{-\alpha}e^{-\beta(t-t_{1})}\left(\E\Phi_{t}\right)^{1-\frac{2}{q}}n^{\frac{2}{q}}\\
 & =\beta\left(\frac{t}{t_{1}}\right)^{-\frac{2}{q}\alpha}e^{-\frac{2}{q}\beta(t-t_{1})}\Psi_{t}^{1-\frac{2}{q}}n^{\frac{2}{q}}.
\end{align*}
Hence, 
\[
\frac{d}{dt}\Psi_{t}^{2/q}\leq\frac{2\beta}{q}\left(\frac{t}{t_{1}}\right)^{-\frac{2}{q}\alpha}e^{-\frac{2}{q}\beta(t-t_{1})}n^{\frac{2}{q}}.
\]
Integrating between $t_{1}$and $t$,
\[
\Psi_{t}^{2/q}\leq\Psi_{t_{1}}^{2/q}+\frac{2\beta n^{\frac{2}{q}}}{q}\int_{t_{1}}^{t}\left(\frac{s}{t_{1}}\right)^{-\frac{2}{q}\alpha}ds\leq\Psi_{t_{1}}^{2/q}+\frac{2\beta n^{\frac{2}{q}}}{q}\frac{t_{1}}{\frac{2}{q}\alpha-1}.
\]
Using $\alpha=q(q-1)$ and $q\geq3$, we have $\Psi_{t}^{2/q}\leq\Psi_{t_{1}}^{2/q}+\frac{2\beta n^{2/q}}{q^{2}}t_{1}$.
Using the definition of $\Psi_{t}$, we have
\[
\E\Phi_{t}\leq\left(\frac{t}{t_{1}}\right)^{\alpha}e^{\beta(t-t_{1})}\left(\Psi_{t_{1}}^{2/q}+\frac{2\beta n^{\frac{2}{q}}}{q^{2}}t_{1}\right)^{q/2}=\left(\frac{t}{t_{1}}\right)^{\alpha}e^{\beta(t-t_{1})}\left((\E\Phi_{t_{1}})^{2/q}+\frac{2\beta n^{\frac{2}{q}}}{q^{2}}t_{1}\right)^{q/2}.
\]
For all $t\leq\kappa_{n}^{-2}$, using $q\leq4$, $t_{1}=\frac{1}{c\kappa_{n}^{2}\log n}$,
$\beta=C\kappa_{n}^{2}$ and Lemma \ref{lem:trAq}, we have
\[
\E\Phi_{t}\lesssim\left(\frac{t}{t_{1}}\right)^{\alpha}\left(\E\Phi_{t_{1}}+\frac{n}{\log^{q/2}n}\right)\lesssim\left(\frac{t}{t_{1}}\right)^{q(q-1)}\frac{n}{\log^{q/2}n}.
\]
In particular for $t\leq\hat{t}$ this gives
\[
\E\Phi_{t}\lesssim\left(\log^{\frac{1}{2q-2}}n\right)^{q(q-1)}\frac{n}{\log^{q/2}n}=n.
\]
Thus $\E\tr A_{t}^{q}\lesssim2^{q}\left(\E\Phi_{t}+n\right)\lesssim n$
for $t\leq\hat{t}$. For $t\geq\hat{t}$, we instead use Lemma \ref{lem:Aq_growth}
with initial distribution given by $p_{\hat{t}}$. Taking expectations
implies
\[
\E\tr A_{t}^{q}\lesssim\left(\frac{t}{\hat{t}}\right)^{q(q-1)}\E\tr A_{\hat{t}}^{q}\lesssim\left(\frac{t}{\hat{t}}\right)^{q(q-1)}n\lesssim\left(\frac{t}{t_{1}}\right)^{q(q-1)}\frac{n}{\log^{q/2}n}
\]
where the last inequality follows from $\hat{t}=t_{1}\log^{\frac{1}{2q-2}}n$.
\end{proof}
We use the above bound on $\tr A_{t}^{q}$ to bound $\E\tr A_{t}^{2}$
and hence $\E\|a_{t}\|^{2}$.
\begin{lem}
\label{lem:mu2_late}Under the setting of Lemma \ref{lem:trAq}. For
any $3\leq q\leq4$, we have
\[
\E\|a_{t}\|^{2}\lesssim nt+n\frac{t^{\gamma+1}}{(t^{*})^{\gamma}}
\]
where $t^{*}\defeq\kappa_{n}^{-2\frac{q^{2}-q}{q^{2}-2}}\log^{-\frac{2q^{2}-3q}{2(q^{2}-2)}}n$
and $\gamma\leq2\sqrt{2}$.
\end{lem}

\begin{proof}
Using $dA_{t}=\E_{x\sim p_{t}}(x-a_{t})(x-a_{t})^{\top}(x-a_{t})^{\top}dW_{t}-A_{t}^{2}dt$
(\ref{lem:LV_property}) and Itô's formula, we have
\[
d\tr A_{t}^{2}=2\E_{x\sim p_{t}}(x-a_{t})^{\top}A_{t}(x-a_{t})\cdot(x-a_{t})^{\top}dW_{t}-2\tr A_{t}^{3}dt+T_{\mu_{t}}(I,I,I)dt.
\]
Taking expectation and using Lemma \ref{lem:TIII_conditional}, we
have
\[
\frac{d}{dt}\E\tr A_{t}^{2}\leq\E T_{\mu_{t}}(I,I,I)\lesssim\frac{1}{t^{3}}\E(t^{q}\tr A_{t}^{q})^{c}(t^{2}\tr A_{t}^{2})^{1-c}
\]
with $c=\frac{1}{2(q^{2}-2)}$. Applying $\E\tr A_{t}^{q}\lesssim n+(\frac{t}{t_{1}})^{q(q-1)}\frac{n}{\log^{q/2}n}$
(Lemma \ref{lem:trAq_late}), we have
\[
\frac{d}{dt}\E\tr A_{t}^{2}\lesssim t^{(q-2)c-1}n^{c}(\E\tr A_{t}^{2})^{1-c}+t^{-1+(q-2)c+q(q-1)c}\left(\frac{1}{t_{1}^{q(q-1)}}\frac{n}{\log^{q/2}n}\right)^{c}(\E\tr A_{t}^{2})^{1-c}.
\]
Rearranging and using $c=\Theta(1)$ (because $3\leq q\leq4$), we
have
\[
\frac{d}{dt}(\E\tr A_{t}^{2})^{c}\lesssim t^{(q-2)c-1}n^{c}+t^{-1+(q^{2}-2)c}\left(\frac{1}{t_{1}^{q(q-1)}}\frac{n}{\log^{q/2}n}\right)^{c}.
\]
Integrating from $t_{1}$ and using $\E\tr A_{t_{1}}^{2}\lesssim n$,
for all $t\geq t_{1}$
\[
(\E\tr A_{t}^{2})^{c}\lesssim n^{c}+t^{(q-2)c}n^{c}+t^{(q^{2}-2)c}\left(\frac{1}{t_{1}^{q(q-1)}}\frac{n}{\log^{q/2}n}\right)^{c}.
\]
Hence, we have
\begin{equation}
\E\tr A_{t}^{2}\lesssim n+nt^{q-2}+\frac{t^{q^{2}-2}}{t_{1}^{q^{2}-q}}\frac{n}{\log^{q/2}n}.\label{eq:trA2_final}
\end{equation}
In particular, this shows that for $0\leq t\leq t^{*}$ with $t^{*}\defeq t_{1}^{\frac{q^{2}-q}{q^{2}-2}}\log^{\frac{q}{2(q^{2}-2)}}n$,
we have $\E\tr A_{t}^{2}\lesssim n$. 

For $t\geq t^{*}$, we switch to the estimate $T_{\mu_{t}}(I,I,I)\leq(\gamma/t)\tr A_{t}^{2}$
with $\gamma\leq2\sqrt{2}$ (Lemma \ref{lem:TIII}). Applying it to
$\frac{d}{dt}\E\tr A_{t}^{2}\leq\E T_{\mu_{t}}(I,I,I)$, we have
\[
\E\tr A_{t}^{2}\lesssim\left(\frac{t}{t^{*}}\right)^{\gamma}\E\tr A_{t^{*}}^{2}\lesssim\left(\frac{t}{t^{*}}\right)^{\gamma}n
\]
for $t\geq t^{*}$. Hence, for all $t>0$, we have 
\[
\E\tr A_{t}^{2}\lesssim\left(1+\left(\frac{t}{t^{*}}\right)^{\gamma}\right)n.
\]

Using $da_{t}=A_{t}dW_{t}$ (Lemma \ref{lem:LV_property}) and Itô's
formula, we have $\frac{d}{dt}\E\|a_{t}\|^{2}=\tr A_{t}^{2}$. Hence,
we have
\[
\E\|a_{t}\|^{2}\lesssim\int_{0}^{t}(1+\left(\frac{s}{t^{*}}\right)^{\gamma})n\cdot ds\lesssim nt+n\frac{t^{\gamma+1}}{(t^{*})^{\gamma}}
\]
with $t^{*}\defeq\kappa_{n}^{-2\frac{q^{2}-q}{q^{2}-2}}\log^{-\frac{q^{2}-q}{q^{2}-2}+\frac{q}{2(q^{2}-2)}}n=\kappa_{n}^{-2\frac{q^{2}-q}{q^{2}-2}}\log^{-\frac{2q^{2}-3q}{2(q^{2}-2)}}n$.
This gives the result.
\end{proof}

\section{Improved bounds for thin shell}
\begin{thm}
\label{thm:main_result}We have $\sigma_{n}\lesssim\log^{\eta}n$
for 
\[
\eta=\min_{3\leq q\leq4}\frac{1+\frac{q^{2}-\frac{5}{4}q}{q^{2}-2}\gamma}{1+\frac{q-2}{q^{2}-2}\gamma}.
\]
Setting $\gamma=2\sqrt{2}$ and $q=\frac{1}{47}(112-16\sqrt{2}+\sqrt{5630-1892\sqrt{2}})$,
we have
\[
\eta\leq\frac{1}{8}(1+7\sqrt{2}+\sqrt{53-4\sqrt{2}})\leq2.2226
\]
\end{thm}

\begin{proof}
Pick an isotropic log-concave distribution $p_{0}$ on $\Rn$ such
that $\sigma_{p_{0}}^{2}\geq\frac{1}{2}\sigma_{n}^{2}$ and $p_{0}$
is supported in a ball of radius $n$. Let $p_{t}$ be given by the
LV process with initial distribution $p_{0}$. By Lemma \ref{lem:KL_key}
and Lemma \ref{lem:mu2_late}, we have
\[
\sigma_{n}^{2}\lesssim\int_{\lambda_{1}}^{\infty}\min_{t_{\lambda}>0}\left\{ \frac{1}{n\lambda^{2}}\E\|a_{t_{\lambda}}\|^{2}+\frac{1}{\lambda t_{\lambda}}\right\} d\lambda\lesssim\int_{\lambda_{1}}^{\infty}\min_{t_{\lambda}>0}\left\{ \frac{t_{\lambda}}{\lambda^{2}}+\frac{t_{\lambda}^{\gamma+1}}{\lambda^{2}(t^{*})^{\gamma}}+\frac{1}{\lambda t_{\lambda}}\right\} d\lambda
\]
with $\lambda_{1}\approx\psi_{p_{0}}^{-2}$. Picking $t_{\lambda}=\min(\sqrt{\lambda},\lambda^{\frac{1}{\gamma+2}}(t^{*})^{\frac{\gamma}{\gamma+2}})$,
we have
\[
\sigma_{n}^{2}\lesssim\int_{\lambda_{1}}^{\infty}\frac{1}{\lambda^{1.5}}+\frac{1}{\lambda^{1+\frac{1}{\gamma+2}}(t^{*})^{\frac{\gamma}{\gamma+2}}}d\lambda\lesssim\frac{1}{\sqrt{\lambda_{1}}}+\frac{1}{\lambda_{1}^{\frac{1}{\gamma+2}}(t^{*})^{\frac{\gamma}{\gamma+2}}}.
\]
Using $t^{*}=\kappa_{n}^{-2\frac{q^{2}-q}{q^{2}-2}}\log^{-\frac{2q^{2}-3q}{2(q^{2}-2)}}n$
(Lemma \ref{lem:mu2_late}), $\psi_{n}^{2}\lesssim\sigma_{n}^{2}\log^{2}n$
and $\kappa_{n}^{2}\lesssim\sigma_{n}^{2}\log n$ (Lemma \ref{lem:thin_shell_gap}),
we have 
\begin{align*}
\sigma_{n}^{2} & \lesssim\sigma_{n}\log n+\sigma_{n}^{\frac{2}{\gamma+2}}\log^{\frac{2}{\gamma+2}}n\cdot(\kappa_{n}^{-2\frac{q^{2}-q}{q^{2}-2}}\log^{-\frac{2q^{2}-3q}{2(q^{2}-2)}}n)^{-\frac{\gamma}{\gamma+2}}\\
 & \lesssim\sigma_{n}\log n+\sigma_{n}^{\frac{2}{\gamma+2}}\log^{\frac{2}{\gamma+2}}n\cdot(\sigma_{n}^{-2\frac{q^{2}-q}{q^{2}-2}}\log^{-\frac{q^{2}-q}{q^{2}-2}}n\log^{-\frac{2q^{2}-3q}{2(q^{2}-2)}}n)^{-\frac{\gamma}{\gamma+2}}\\
 & =\sigma_{n}\log n+\sigma_{n}^{2(\frac{q^{2}-q}{q^{2}-2}\frac{\gamma}{\gamma+2}+\frac{1}{\gamma+2})}\log^{\frac{2}{\gamma+2}(1+\frac{4q^{2}-5q}{4(q^{2}-2)}\gamma)}n.
\end{align*}
Relabel the exponents as $\alpha=2\left(\frac{q^{2}-q}{q^{2}-2}\frac{\gamma}{\gamma+2}+\frac{1}{\gamma+2}\right)$
and $\beta=\frac{2}{\gamma+2}\left(1+\frac{4q^{2}-5q}{4(q^{2}-2)}\gamma\right)$.
Let $\eta$ be the smallest value such that $\sigma_{n}\leq C\log^{\eta}n$
for an absolute constant $C$. The above inequality implies

\[
2\eta\leq\max\{\eta+1,\alpha\eta+\beta\}.
\]
Note that $\alpha,\beta\geq1$ for all $\gamma\geq0$: thus $\alpha\eta+\beta\geq\eta+1.$
Rearranging and substituting the values of $\alpha$ and $\beta$
gives

\[
\eta\leq\frac{\beta}{2-\alpha}=\frac{1+\frac{4q^{2}-5q}{4(q^{2}-2)}\gamma}{1+\frac{q-2}{q^{2}-2}\gamma}.
\]
This shows that $\sigma_{n}\lesssim\log^{\eta}n$ for the given value
of $\gamma$ and any $3\leq q\leq4$. The last claim follows by setting
$\gamma=2\sqrt{2}$ and optimizing $q$.
\end{proof}
\begin{thm}
\label{thm:no_gap} If $\psi_{n}\lesssim\sigma_{n}$ , then we have
$\sigma_{n}\lesssim\log^{\eta}n$ with
\[
\eta=\min_{3\leq q\leq4}\frac{\frac{\frac{1}{2}q^{2}-\frac{3}{4}q}{q^{2}-2}\gamma}{1+\frac{q-2}{q^{2}-2}\gamma}.
\]
Setting $\gamma=2\sqrt{2}$ and $q=3$, we have $\eta\leq\frac{63\sqrt{2}-36}{82}\leq0.6476$.
\end{thm}

\begin{proof}
The proof is similar to Theorem \ref{thm:main_result}. The key difference
is that we have $\psi_{n}^{2}\lesssim\sigma_{n}^{2}$ and $\kappa_{n}^{2}\lesssim\sigma_{n}^{2}$.
Hence, the same calculation above shows that
\begin{align*}
\sigma_{n}^{2} & \lesssim\sigma_{n}+\sigma_{n}^{\frac{2}{\gamma+2}}\cdot(\kappa_{n}^{-2\frac{q^{2}-q}{q^{2}-2}}\log^{-\frac{2q^{2}-3q}{2(q^{2}-2)}}n)^{-\frac{\gamma}{\gamma+2}}\\
 & \lesssim\sigma_{n}+\sigma_{n}^{\frac{2}{\gamma+2}}\cdot(\sigma_{n}^{-2\frac{q^{2}-q}{q^{2}-2}}\log^{-\frac{2q^{2}-3q}{2(q^{2}-2)}}n)^{-\frac{\gamma}{\gamma+2}}\\
 & =\sigma_{n}+\sigma_{n}^{2(\frac{q^{2}-q}{q^{2}-2}\frac{\gamma}{\gamma+2}+\frac{1}{\gamma+2})}\log^{\frac{q^{2}-\frac{3}{2}q}{q^{2}-2}\frac{\gamma}{\gamma+2}}n.
\end{align*}
Again, define $\alpha=2(\frac{q^{2}-q}{q^{2}-2}\frac{\gamma}{\gamma+2}+\frac{1}{\gamma+2})$
and $\beta=\frac{q^{2}-\frac{3}{2}q}{q^{2}-2}\frac{\gamma}{\gamma+2}$:
note $\alpha\geq1.$ If $\eta$ is the smallest value where $\sigma_{n}\leq C\log^{\eta}n$
for an absolute constant $C,$the above bound yields
\[
\eta\leq\frac{\beta}{2-\alpha}=\frac{\frac{\frac{1}{2}q^{2}-\frac{3}{4}q}{q^{2}-2}\gamma}{1+\frac{q-2}{q^{2}-2}\gamma}.
\]
This shows that $\sigma_{n}\lesssim\log^{\eta}n$ for any $3\leq q\leq4.$
The last result follows from setting $\gamma=2\sqrt{2}$ and $q=3$.
We remark that for the current value of $\gamma,$the value of $q$
minimizing the above bound is at $q=2.4588$. Unfortunately, Lemma
\ref{lem:mu2_late} does not apply for $q<3$.
\end{proof}

\paragraph*{Acknowledgement. }

This work was supported in part by NSF awards CCF-2007443 and CCF-2134105
and a Packard fellowship.

\appendix
\bibliographystyle{plain}
\bibliography{acg}

\begin{thebibliography}{10}

\bibitem{chen2021almost}
Yuansi Chen.
\newblock An almost constant lower bound of the isoperimetric coefficient in
  the kls conjecture.
\newblock {\em Geometric and Functional Analysis}, 31(1):34--61, 2021.

\bibitem{cordero2004b}
Dario Cordero-Erausquin, Matthieu Fradelizi, and Bernard Maurey.
\newblock The (b) conjecture for the gaussian measure of dilates of symmetric
  convex sets and related problems.
\newblock {\em Journal of Functional Analysis}, 214(2):410--427, 2004.

\bibitem{Eldan2013}
R.~Eldan.
\newblock Thin shell implies spectral gap up to polylog via a stochastic
  localization scheme.
\newblock {\em Geometric and Functional Analysis}, 23:532--569, 2013.

\bibitem{eldan2014bounding}
Ronen Eldan and Joseph Lehec.
\newblock Bounding the norm of a log-concave vector via thin-shell estimates.
\newblock In {\em Geometric Aspects of Functional Analysis}, pages 107--122.
  Springer, 2014.

\bibitem{jia2021reducing}
He~Jia, Aditi Laddha, Yin~Tat Lee, and Santosh Vempala.
\newblock Reducing isotropy and volume to kls: an o*(n 3 $\psi$ 2) volume
  algorithm.
\newblock In {\em Proceedings of the 53rd Annual ACM SIGACT Symposium on Theory
  of Computing}, pages 961--974, 2021.

\bibitem{jiang2020generalized}
Haotian Jiang, Yin~Tat Lee, and Santosh~S Vempala.
\newblock A generalized central limit conjecture for convex bodies.
\newblock In {\em Geometric Aspects of Functional Analysis}, pages 1--41.
  Springer, 2020.

\bibitem{juditsky2008large}
Anatoli Juditsky and Arkadii~S Nemirovski.
\newblock Large deviations of vector-valued martingales in 2-smooth normed
  spaces.
\newblock {\em arXiv preprint arXiv:0809.0813}, 2008.

\bibitem{KLS95}
R.~Kannan, L.~Lov{\'a}sz, and M.~Simonovits.
\newblock Isoperimetric problems for convex bodies and a localization lemma.
\newblock {\em Discrete \& Computational Geometry}, 13:541--559, 1995.

\bibitem{klartag2022bourgain}
Bo'az Klartag and Joseph Lehec.
\newblock Bourgain's slicing problem and kls isoperimetry up to polylog.
\newblock {\em arXiv preprint arXiv:2203.15551v2}, 2022.

\bibitem{LV_survey2019}
Y.~T. Lee and S.~Vempala.
\newblock The kannan-lovasz-simonovits conjecture.
\newblock {\em Current developments in mathematics (2017)}, pages 1--36, 2019.

\bibitem{lee2016eldan}
Yin~Tat Lee and Santosh~S Vempala.
\newblock Eldan's stochastic localization and the kls conjecture: Isoperimetry,
  concentration and mixing.
\newblock {\em arXiv preprint arXiv:1612.01507v3}, 2016.

\end{thebibliography}

\section{Discussion on possible value of $\gamma$\label{sec:gamma}}

For a $1$-dimensional $1$-strongly log-concave distribution $\mu$,
we have
\[
T_{\mu}(I,I,I)=\E_{x,y\sim\mu}(xy)^{3}=(\E_{x\sim\mu}x^{3})^{2}.
\]
Therefore, the inequality $T_{\mu}(I,I,I)\leq\gamma\tr A^{2}$ becomes
\[
(\E_{x\sim\mu}x^{3})^{2}\leq\gamma\left(\E_{x\sim\mu}x^{2}\right)^{2}.
\]
Note that $(\E_{x\sim\mu}x^{3})^{2}-\gamma\left(\E_{x\sim\mu}x^{2}\right)^{2}$
is convex in $\mu$. Hence, the localization lemma shows that the
$1$-strongly log-concave distribution that maximizes $(\E_{x\sim\mu}x^{3})^{2}-\gamma\left(\E_{x\sim\mu}x^{2}\right)^{2}$
is a truncated Gaussian distribution (with variance $1$ before truncation).

A computer search suggests that $\gamma\approx0.37$ with the worst
distribution is 
\[
\exp(-\frac{(x-m)^{2}}{2})1_{[a,b]}
\]
with $a\approx-0.69,b\approx4.31,m\approx-1.03$.

\section{Deferred calculations \label{sec:defer_cal}}

In this section, we include various deferred calculations. All of
them are variants of existing proofs. 

The following lemma is an improvement of Lemma \ref{lem:TAII_chen}.
It is given in \cite[Remark 4.4]{klartag2022bourgain}. We give a
proof here for completeness.
\begin{lem}
\label{lem:TAII}For any $\alpha$-strongly log-concave $\mu$ with
covariance $A$ and any $q\geq3$, we have
\[
T_{\mu}(A^{q-2},I,I)\leq\frac{2}{\alpha}\tr A^{q}.
\]
\end{lem}

\begin{proof}
By centering, we can assume $\mu$ has mean $0$. Let $\eta_{i}$
and $v_{i}$ are eigenvalues and eigenvectors of $A$ and hence $A=\sum_{i=1}^{d}\eta_{i}v_{i}v_{i}^{\top}$.
Let 
\[
\Delta_{i}=\E_{x\sim\mu}(x^{\top}A^{-1/2}v_{i})\cdot xx^{\top}.
\]
Then, we have
\begin{align}
T_{\mu}(A^{q-2},I,I)= & \E_{x,y\sim\mu}x^{\top}A^{q-2}y\cdot x^{\top}y\cdot x^{\top}y\nonumber \\
= & \sum_{i=1}^{d}\eta_{i}^{q-1}\E_{x,y\sim\mu}x^{\top}A^{-1/2}v_{i}v_{i}A^{-1/2}y\cdot x^{\top}y\cdot x^{\top}y\nonumber \\
= & \sum_{i=1}^{d}\eta_{i}^{q-1}\tr(\Delta_{i}^{2}).\label{eq:lem:TAII_1}
\end{align}
For the term $\tr(\Delta_{i}^{2})$, using the fact that $x$ has
mean $0$ and variance $A$, we have
\begin{align}
\tr(\Delta_{i}^{2}) & =\E_{x\sim\mu}x^{\top}A^{-1/2}v_{i}\cdot x^{\top}\Delta_{i}x\nonumber \\
 & \leq\sqrt{\E_{x\sim\mu}(x^{\top}A^{-1/2}v_{i})^{2}}\sqrt{\var(x^{\top}\Delta_{i}x)}\nonumber \\
 & =\sqrt{\var(x^{\top}\Delta_{i}x)}.\label{eq:lem:TAII_2}
\end{align}
Since $\mu$ is $\alpha$-strongly log-concave and the gradient of
$x^{\top}\Delta_{i}x$ has mean $0$, Lemma \ref{lem:half_poincare}
shows that
\[
\var(x^{\top}\Delta_{i}x)\leq\frac{1}{2\alpha}\E\|2\Delta_{i}x\|_{2}^{2}\leq\frac{2}{\alpha}\tr(A\Delta_{i}^{2}).
\]
Using this, (\ref{eq:lem:TAII_1}) and (\ref{eq:lem:TAII_2}), we
have
\begin{align*}
T_{\mu}(A^{q-2},I,I) & \leq\sum_{i=1}^{d}\eta_{i}^{q-1}\sqrt{\frac{2}{\alpha}\tr(A\Delta_{i}^{2})}\leq\sqrt{\frac{2}{\alpha}\sum_{i=1}^{d}\eta_{i}^{q}}\sqrt{\sum_{i=1}^{d}\eta_{i}^{q-2}\tr(A\Delta_{i}^{2})}\\
 & =\sqrt{\frac{2}{\alpha}\tr A^{q}}\sqrt{T_{\mu}(A^{q-3},A,I)}\leq\sqrt{\frac{2}{\alpha}\tr A^{q}}\sqrt{T_{\mu}(A^{q-2},I,I)}
\end{align*}
where we used Lemma \ref{lem:T_property} at the end. This gives the
result.
\end{proof}
Next, we compute the Hessian of $\tr(M^{-1/2}A_{t}M^{-1/2}-I)^{q}$.
The proof is similar to \cite[Lemma B.5]{jia2021reducing}. The proof
uses the following estimate of the Hessian of matrix functions.
\begin{lem}[{\cite[Proposition 3.1]{juditsky2008large}}]
\label{lem:matrix_Hess}Let $f$ be a twice differentiable function
on $(\alpha,\beta)$ such that for some $\theta,\varphi\in\R$, for
all $\alpha\leq a<b<\beta$, we have
\[
\frac{f'(b)-f'(a)}{b-a}\leq\theta\frac{f''(b)+f''(a)}{2}+\varphi.
\]
Then, for any matrix $X$ with eigenvalues lies between $(\alpha,\beta)$,
we have
\[
\frac{\partial^{2}\tr f(X)}{\partial X^{2}}|_{H,H}\leq\theta\tr(f''(X)H^{2})+\varphi\tr H^{2}.
\]
\end{lem}

Next we need a simple lemma about the function $(x^{+})^{q}$.
\begin{lem}
\label{lem:xq_hess}Let 
\[
\phi(x)=\begin{cases}
x^{q} & \text{if }x\geq0\\
0 & \text{else}
\end{cases}
\]
 with $q\geq3$. For any $a<b$, we have 
\[
\frac{\phi'(b)-\phi'(a)}{b-a}\leq\frac{\phi''(b)+\phi''(a)}{2}.
\]
\end{lem}

\begin{proof}
\textbf{Case 1)} $0\leq a<b$. 

For a fixed $a$, we let $V(b)=\frac{\phi''(a)+\phi''(b)}{2}(b-a)-(\phi'(b)-\phi'(a))$.
Note that $V(a)=0$ and that
\[
V(b)=q(q-1)\frac{b^{q-2}+a^{q-2}}{2}(b-a)-q(b^{q-1}-a^{q-1}).
\]
Taking derivative with respect to $b$, we have that for all $b$
\[
\frac{2V'(b)}{q(q-1)(q-2)}=\frac{q-3}{q-2}b^{q-2}+\frac{1}{q-2}a^{q-2}-b^{q-3}a\geq0
\]
where we used Young's inequality at the end. Hence, we have $V'(b)\geq0$
for all $b>0$. Hence, $V$ is increasing and that $V(b)\geq V(a)=0$
for all $b\geq a$. 

\textbf{Case 2)} $a\leq0\leq b$.

Since $q\geq3$, we have 
\[
\frac{\phi'(b)-\phi'(a)}{b-a}=\frac{qb^{q-1}}{b}\leq\frac{q(q-1)}{2}b^{q-2}=\frac{\phi''(b)+\phi''(a)}{2}.
\]

\textbf{Case 3)} $a\leq b\leq0$.

Both sides are $0$.
\end{proof}
We can now compute the Hessian.
\begin{lem}
\label{lem:trQ_hess}Let $\mu$ be a log-concave distribution in $\Rn$
with covariance $A$. Let $p_{t}$ be the result of the LV process
starting with the density of $\mu$. Let $A_{t}$ be the covariance
of $p_{t}$ and $\phi$ be as defined in Lemma \ref{lem:xq_hess}.
For any $q\in\{2\}\cup[3,+\infty)$, we have
\begin{align*}
d\tr\phi(A_{t}-bI)= & \E_{x\sim p_{t}}(x-a_{t})^{\top}\phi'(A_{t}-bI)(x-a_{t})(x-a_{t})^{\top}dW_{t}\\
 & -\tr(\phi'(A_{t}-bI)A_{t}^{2})dt+\frac{1}{2}T_{p_{t}}(\phi''(A_{t}-bI),I,I)dt.
\end{align*}
In particular, 
\begin{align*}
\frac{d}{dt}\E\tr\phi(A_{t}-bI)\leq & \frac{1}{2}T_{p_{t}}(\phi''(A_{t}-bI),I,I)
\end{align*}
where the expectation conditional on $A_{t}$.
\end{lem}

\begin{proof}
Let $\Phi_{t}=\tr\phi(A_{t}-bI)$. By Lemma \ref{lem:LV_property},
we have 
\begin{align*}
dA_{t} & =\E_{x\sim p_{t}}(x-a_{t})(x-a_{t})^{\top}(x-a_{t})^{\top}dW_{t}-A_{t}^{2}dt=\sum_{i}Z_{i}\cdot dW_{t,i}-A_{t}^{2}dt
\end{align*}
where $Z_{i}=\E_{x\sim p_{t}}(x-a_{t})(x-a_{t})^{\top}(x-a_{t})_{i}$.
By Itô's formula, we have
\[
d\Phi_{t}=\left.\frac{\partial\tr\Phi_{t}}{\partial A_{t}}\right|_{dA_{t}}+\frac{1}{2}\sum_{i}\left.\frac{\partial^{2}\tr\Phi_{t}}{\partial A_{t}^{2}}\right|_{Z_{i},Z_{i}}dt.
\]
For the first-order term, we have
\begin{align*}
\left.\frac{\partial\tr\Phi_{t}}{\partial A_{t}}\right|_{dA_{t}}= & \tr\phi'(A_{t}-bI)dA_{t}\\
= & \E_{x\sim p_{t}}(x-a_{t})^{\top}\phi'(A_{t}-bI)(x-a_{t})(x-a_{t})^{\top}dW_{t}-\tr(\phi'(A_{t}-bI)A_{t}^{2})dt.
\end{align*}
For the second-order term, we use Lemma \ref{lem:matrix_Hess} with
$\theta_{+}=1$ and $\varphi_{+}=0$ to get
\[
\left.\frac{\partial^{2}\tr\Phi_{t}}{\partial A_{t}^{2}}\right|_{Z_{i},Z_{i}}\leq\sum_{i}\tr(\phi''(A_{t}-bI)Z_{i}^{2})=T_{p_{t}}(\phi''(A_{t}-bI),I,I).
\]
Combining the first- and second-order terms proves the first claim.

In particular, we have
\begin{align*}
\frac{d}{dt}\E\tr\phi(A_{t}-bI)= & -\tr(\phi'(A_{t}-bI)A_{t}^{2})+\frac{1}{2}T_{p_{t}}(\phi''(A_{t}-bI),I,I).
\end{align*}
Since $\phi'(x)\geq0$ for all $x$, we have
\[
\frac{d}{dt}\E\tr\phi(A_{t}-bI)\leq\frac{1}{2}T_{p_{t}}(\phi''(A_{t}-bI),I,I).
\]
\end{proof}

\end{document}